\documentclass{birkjour}
\usepackage[noadjust]{cite}
\usepackage{xcolor}
\usepackage{hyperref}
\RequirePackage[all]{xy}

\newtheorem{thm}{Theorem}[section]

\newtheorem{lem}[thm]{Lemma}

\theoremstyle{definition}

\theoremstyle{remark}
\newtheorem{rem}[thm]{Remark}

\numberwithin{equation}{section}

\def\cl#1#2{{{{C}\!{\ell}}_{{#1},{#2}}}}
\def\dl#1#2{{{{K}}_{{#1},{#2}}}}
\def\conj#1#2{{{#1}^{(#2)}}}
\def\proj#1#2{{\langle{#1}\rangle}_{#2}}
\def\dlsub#1#2#3{{{{K}}^{#3}_{{#1},{#2}}}}

\def\va{\triangle}

\def\firstn#1{\{1,2,\hdots,{#1}\}}
\def\firstnwithzero#1{\{0,1,2,\hdots,{#1}\}}

\def\Mat{{\rm Mat}}

\def\R{{\mathbb R}}
\def\C{{\mathbb C}}
\def\H{{\mathbb H}}
\def\N{{\mathbb N}}
\def\Z{\mathbb{Z}}
\begin{document}

%
%
%
%
%
%
%
%

\title[Commutative Analogues of Clifford Algebras]{On Commutative Analogues of Clifford \\ Algebras and Their Decompositions}

\begin{abstract}
We investigate commutative analogues of Clifford algebras --- algebras whose generators square to $\pm1$ but commute, instead of anti-commuting as they do in Clifford algebras. We observe that commutativity allows for elegant results. We note that these algebras generalise multicomplex spaces --- we show that a commutative analogue of Clifford algebra is either isomorphic to a multicomplex space or to `multi split-complex space' (space defined just like multicomplex numbers but uses split-complex numbers instead of complex numbers). We do a general study of commutative analogues of Clifford algebras and use tools like operations of conjugation and idempotents to give a tensor product decomposition and a direct sum decomposition for them. Tensor product decomposition follows relatively easily from the definition. For the direct sum decomposition, we give explicit basis using new techniques.
\end{abstract}

\author[H. Sharma]{Heerak Sharma}
\address{
Indian Institute of Science Education and Research (IISER), Pune\\
Pune\\
India}

\email{heerak.sharma@students.iiserpune.ac.in}

\author[D. Shirokov]{Dmitry Shirokov}
\address{
HSE University\\
Moscow 101000\\
Russia
\medskip}
\address{
and
\medskip}
\address{
Institute for Information Transmission Problems of the Russian Academy of Sciences \\
Moscow 127051 \\
Russia}
\email{dm.shirokov@gmail.com}


\date{\today}

\label{page:firstblob}
\maketitle

\section{Introduction}
The generators of Clifford algebras anti-commute. We find it natural to consider an analogue of Clifford algebras in which the generators commute. The precise definition of commutative analogues of Clifford algebras is given in Section \ref{section: defns}. N. G. Marchuk in \cite{MarchukDemo, MarchukClassification} looked at tensor product decomposition of Clifford algebras and classified tensor products of Clifford algebras into a few categories and required the need for commutative analogues of Clifford algebras. G. S. Staples, in his books \cite{StaplesBook1,StaplesBook2}, introduces symmetric Clifford algebras, which are exactly commutative analogues of Clifford algebras. He also introduces zeons, which can be thought of as commutative analogues of the exterior (Grassmann) algebras $\cl{0,0}{n}$, and uses them to solve combinatorial problems. Study of commutative analogues of Clifford algebras can aid study of Clifford algebras because one can realise commutative analogues of Clifford algebras as subalgebras of some Clifford algebras. Check Proposition 2.2 in \cite{StaplesBook2}. In particular, multicomplex spaces can be thought of as a subalgebra of Clifford algebras, check Proposition 3.15 in \cite{MSThesis}. Multicomplex spaces and commutative quaternions are two well-known examples of commutative analogues of Clifford algebras.
 
Multicomplex spaces were first introduced by Segre in 1892 in his seminal paper \cite{Segre}. They have been studied extensively and are rich with beautiful results \cite{Price, Shapiro, MulticomplexIdeals, HolomorphyInMC, MulticomplexHyperfunctions, DifferentialEQnInMC}. The most well-known example of multicomplex spaces are the bicomplex numbers which have been studied extensively \cite{BCHolFunction, LunaBCAndElFun,FuncAnalBC,SingularititesBCHolFunc, GeometryIdentitityTheorems, FractalTypeSets, CauchyTypeIntegral, SingulFuncBCVar, BCHolFunCal, BicomplexHyperfunctions}. More interest has been present towards studying analysis over multicomplex spaces. Bicomplex spaces and multicomplex spaces in general find applications in a lot of areas, for instance check \cite{SlicesOfMandelbrodSet,FrenchMSThesis,MSThesis,PlatonicSolids} for applications in studying Mandelbrod set, \cite{BCQM1,BCQM2,MCSchrodingerEQn} for applications in quantum physics and \cite{FinDImBCHilbertSpace,BCRiemannZetaFunc,MCRiemannZetaFunc,BCPolygammaFunction} for some other applications. In Section \ref{discussion}, we discuss commutative analogues of Clifford algebras, multicomplex spaces and related structures and discuss how some results from this paper may aid study of multicomplex spaces.

Commutative quaternions are isomorphic to bicomplex numbers and have been well studied. F. Catoni et al. in \cite{Catoni1} and \cite{Catoni2} discussed commutative quaternions and in general some commutative hypercomplex numbers. H.H. Kösal and M. Tosun in \cite{CommQuatMat} considered matrices with entries from commutative quaternions. Commutative quaternions find applications in a lot of areas, for instance check \cite{ZhangWangVasilevSVD, ZhangJiangVasilevLeastSQ, ZhangGuoWangJiangLeastSQ, ZhangWangVasilevJiangLeastSQ, Isokawa, Kobayashi} for applications in computer science and \cite{Catoni3, ZhangWangVasilevJiangMaxwellEqn} for applications in physics. 

This paper is organised as follows. In Section~\ref{section: defns}, we define the commutative analogues of Clifford algebras and introduce the notations that we will use in this paper. In Section~\ref{section: tpd}, we give a tensor product decomposition for them. In Section~\ref{section: isomorphism classes}, we discuss isomorphisms between different commutative analogues of Clifford algebras and introduce the idea of isomorphism classes. We show that multicomplex spaces are one of the isomorphism classes. In Section~\ref{Operations of conjugation}, we discuss operations of conjugation in these algebras. In Section~\ref{Section: Special basis in K(p,q)}, we introduce a special basis in these algebras which leads to a direct sum decomposition as a direct sum of $\R$s or a direct sum of $\C$s. As a consequence, we recover the idempotents in multicomplex spaces. In Section~\ref{discussion}, we discuss how our work might simplify the study of multicomplex spaces. The conclusions follow in Section ~\ref{section_conclusions}. 

We were inspired to do this study when studying multiplicative inverses in Clifford algebras and therefore we present this paper from an algebraist's viewpoint. Theorems \ref{tensor product decomposition theorem}, \ref{thm: isomorphism classes}, \ref{Theorem: isomorphism between K_{n,0} and direct sum of Rs},  \ref{Theorem: isomorphism between K_{0,n} and direct sum of Cs} have been presented for the first time. New techniques regarding idempotents in multicomplex spaces are presented in Section \ref{subsection: special basis in K_{0,n}} as a sequence of lemmas and theorems (Lemma \ref{Lemma: first properties of special basis in K(0,n)} to Theorem \ref{Theorem: isomorphism between K_{0,n} and direct sum of Cs}). These made it possible to present explicit basis for the decompositions. These techniques can be interesting in themselves for the development of the theory.

\section{Definitions, notations and first examples}\label{section: defns}
\subsection{Defining $\dl{p}{q}$}
Let $p,q \in \Z_{\ge0}$ and $n := p+q$. The \textit{commutative analogue of Clifford algebra}, denoted by $\dl{p}{q}$\footnote{We denote commutative analogues of Clifford algebras by $\dl{p}{q}$ -- a notation introduced by N. G. Marchuk in \cite{MarchukClassification}.}, is a real\footnote{One could define the algebra over any field. We, in this paper, consider only algebras over the real numbers.} associative algebra with unity\footnote{An associative algebra $A$ over a field $\mathbb{K}$ is a vector space over $\mathbb{K}$ along with a bilinear, associative multiplication $\cdot: A\times A \to A$. The tag `with unity' means that the algebra admits multiplicative identity.} generated by the generators $\{e_1, e_2, \hdots, e_n\}$ which obey the following multiplication:
\begin{align*}
&{e_i}^2 = 1, \;\; \text{if} \;\; {1\le i \le p};\\
&{e_i}^2 = -1, \;\; \text{if} \;\; {p+1 \le i \le p+q =n};\\
&{e_i}{e_j} = {e_j}{e_i} \;\; \forall \;\; i,j \in \firstn{n}.
\end{align*}
The multiplication is then extended to all of $\dl{p}{q}$ via bi-linearity. Note that because the generators commute, the algebra that they generate, $\dl{p}{q}$, is also commutative.\\

Let $l\in\N$. For $1 \le i_1 < i_2 < \cdots < i_l \le n$ define $e_{i_1 i_2 \ldots i_l} := e_{i_1}e_{i_2} \cdots e_{i_l}$. Also, define $e_{\{\}} := 1$ and $e_{\{i_1, i_2, \ldots, i_l\}} := e_{i_1 i_2 \ldots i_l}$. A general element in $\dl{p}{q}$ is of the form:
\begin{multline*}
    U = u + \sum_{1 \le i \le n} u_ie_i + \!\!\sum_{1 \le i_1 < i_2 \le n}\!\! u_{i_1 i_2}e_{{i_1} {i_2}} + \quad\cdots\quad + \!\!\!\!\!\!\!\!\!\!\\ \sum_{1 \le i_i < i_2 < \cdots < i_{n-1} \le n}\!\!\!\!\!\!\!\!\!\! u_{i_1 i_2 \ldots i_{n-1}}e_{i_1 i_2 \ldots i_{n-1}} + u_{12\ldots n}e_{12\ldots n}.
\end{multline*}

\subsection{First examples} 
\begin{enumerate}
    \item It is easy to check that $\dl{1}{0} \cong \R\oplus\R$ - the split-complex numbers and $\dl{0}{1} \cong \C$ - the complex numbers\footnote{Comment: These are also same as the Clifford algebra $\cl{1}{0}$ and $\cl{0}{1}$. These along with $\cl{0}{0} \cong \R$ are the only Clifford algebras that are commutative.}. 

    \item The algebra $\dl{1}{1}$ is known as the algebra of commutative quaternions \cite{Catoni1, Catoni2, CommQuatMat}.

    \item The subalgebra $S = {\rm{Span}_\R}\{1,e_{1256},e_{1346},e_{2345}\} \subseteq \cl{6}{0}$ described in \cite{A.Acus} is isomorphic to $\dl{2}{0}$. This becomes apparent once we define $g_1 = e_{1256}, g_2 = e_{1346}$. Then $g_1 g_2 = g_2 g_1 = -e_{2345}$ and $g_1^2 = g_2^2 = +1$. Therefore, the subalgebra $S\cong\dl{2}{0}$ the isomorphism being identification of $g_1,g_2$ as generators of $\dl{2}{0}$. 
    
    In general, using Proposition 2.2 in \cite{StaplesBook2}, $\dl{p}{q}$s can be realised as a sub-algebras of the Clifford algebra $\cl{p+2q}{p}$ as mentioned before.
     
    \item It will be shown subsequently in Section \ref{Subsection: Example of tensor product decomposition} that multicomplex spaces are $\dl{0}{n}$ for $n\in\N$.  
\end{enumerate}

\subsection{Grades and projections}
$\dl{p}{q}$ as a vector space is the $2^n$-dimensional space
$$\mathrm{Span}_\R\{e_A \;|\; A \subseteq \firstn{n} \}.$$ The elements $\{e_A \;|\; A \subseteq \firstn{n} \}$ form a basis for $\dl{p}{q}$.

Let $A\subseteq\firstn{n}$. By $|A|$ we denote the number of elements in $A$. For $k \in \firstnwithzero{n}$, we define \textit{subspace of grade $k$}, denoted by $\dlsub{p}{q}{k}$, as 
$$\dlsub{p}{q}{k} := {\rm{Span}}_{\R}\{e_A \;|\; A \subseteq \firstn{n} \text{ and } |A| = k\},$$
i.e., $\dlsub{p}{q}{k}$ is $\R$-$\rm{span}$ of basis elements of length k.

We define \textit{projection maps} $\proj{.}{k}: \dl{p}{q} \to \dlsub{p}{q}{k}$ by $U \mapsto \proj{U}{k}$. These are linear maps which return the grade $k$ part of $U$. In general, 
$$U = \sum_{k=0}^{n}{\proj{U}{k}}.$$
 
\section{Tensor product decomposition}\label{section: tpd}
An immediate consequence of commutativity is that one can decompose any $\dl{p}{q}$ as tensor products of `$p$' $\dl{1}{0}$s and `$q$' $\dl{0}{1}$s. We start by defining tensor product of algebras.

\subsection{Tensor product of algebras}
Let $A$ and $B$ be two associative algebras having multiplicative identities, dimensions, basis as $1_A$ and $1_B$, $n$ and $m$, $\{a_1, \hdots, a_n\}$ and $\{b_1 ,\hdots, b_m\}$ respectively. Because $A$ and $B$ are vector spaces, we consider their tensor product $A\otimes\footnote{We will consider tensor products over the field of real numbers in this paper.} B$ which is the $n\cdot m$-dimensional vector space
$$\mathrm{Span}\{{a_i\otimes b_j \; | \; i \in \firstn{n}, j \in \firstn{m}}\}.$$
One can then make $A\otimes B$ into an associative algebra with identity if one defines the multiplication in $A \otimes B$ `naturally' as 
\begin{equation*}
    (a_{i_1}\otimes b_{j_1})\cdot(a_{i_2} \otimes b_{j_2}) := {a_{i_1}a_{i_2}} \otimes {b_{j_1} b_{j_2}}
\end{equation*} 
for $i_1, i_2 \in \firstn{n}, j_i, j_2 \in \firstn{m}$ and extends this multiplication bi-linearly to all of $A\otimes B$. The identity element of $A\otimes B$ is $1_A\otimes 1_B$.

Using similar idea, we can define tensor product algebra of $n$ algebras $A_1, A_2, \hdots, A_n$ which is $A_1 \otimes A_2\otimes \cdots \otimes A_n$ as a vector space with the multiplication defined as
\begin{equation*}
    (a_1\otimes a_2\otimes \cdots \otimes a_n)\cdot(b_1\otimes b_2\otimes \cdots \otimes b_n) := ({a_1}{b_1})\otimes({a_2}{b_2})\otimes \cdots \otimes ({a_n}{b_n})
\end{equation*}
where $a_i, b_i$ are elements of $A_i$ for all $i \in \firstn{n}$ and then extended bi-linearity to all of $A_1 \otimes A_2\otimes \cdots \otimes A_n$.

Also, it follows that if $A$ and $B$ are commutative, then $A\otimes B$ is also commutative. Let $a\otimes b$, $a^\prime \otimes b^\prime$ $\in A\otimes B$ then:
$$(a\otimes b)\cdot(a^\prime \otimes b^\prime) = {(a a^\prime)} \otimes {(b b^\prime)} = {(a^\prime a)} \otimes {(b^\prime b)} = (a^\prime \otimes b^\prime) \cdot (a\otimes b)$$
showing that if $A$ and $B$ are commutative, then so is $A\otimes B$. Similar argument extends to the tensor product algebra of $n$ algebras i.e., tensor product of $n$ commutative algebras is another commutative algebra. 
Also, since tensor product is also commutative and associative, i.e., $(A \otimes B) \otimes C \cong A \otimes (B \otimes C)$ and $A\otimes B \cong B \otimes A$, it allows us to express the tensor product of $n$ algebras in any order and without need of any parentheses.

\subsection{The tensor product decomposition of $\dl{p}{q}$}\label{Subsection: tensor product decomposition result}

\begin{thm}\label{tensor product decomposition theorem}
    It is the case that:\footnotemark
    \begin{equation}\label{tensor product decomposition result}
        \dl{p}{q} \cong \underbrace{\dl{1}{0} \otimes \dl{1}{0} \cdots \otimes \dl{1}{0}}_{p \; \text{times}} \otimes \underbrace{\dl{0}{1} \otimes \dl{0}{1} \cdots \otimes \dl{0}{1}}_{q \; \text{times}}.
    \end{equation}
\end{thm}

\footnotetext{This decomposition is unique up to permutation of the productants in the tensor product on the right-hand side of the result.}

\noindent This shows that $\dl{p}{q}$ sort of decouples as `$p$' $\dl{1}{0}$s and `$q$' $\dl{0}{1}$s.

It follows from the tensor product decomposition that
\begin{equation}\label{Tensor decomposition Property 1}
\dl{p_1}{q_1}\otimes\dl{p_2}{q_2} \cong \dl{p}{q} \text{ when } p_1 + p_2 = p,\quad q_1 + q_2 = q.
\end{equation}

This follows directly if one decomposes the $\dl{p}{q}$s on the left-hand side and the right-hand side of \eqref{Tensor decomposition Property 1} into tensor products of $\dl{1}{0}$ and $\dl{0}{1}$. This means that tensor product of two commutative analogues of Clifford algebras is another commutative analogue of Clifford algebra. The analogous situation for Clifford algebras is different. Tensor product of two Clifford algebras is not necessarily another Clifford algebra. Please check \cite{MarchukClassification} for more details.

\subsection{An example: multicomplex spaces}\label{Subsection: Example of tensor product decomposition}
Let us consider $\dl{0}{2}$. A general element of $\dl{0}{2}$ is $a + {a_1}{e_1} + {a_2}{e_2} + {a_{12}}{e_{12}}$. Now, note that:
\begin{equation*}
a + {a_1}{e_1} + {a_2}{e_2} + {a_{12}}{e_{12}} = (a + {a_1}{e_1}) + {e_2}({a_2} + {a_{12}}{e_1}) = x + y{e_2}
\end{equation*}
where $x,y \in \dl{0}{1}$. This is exactly the definition of bi-complex numbers \cite{Price}!\footnote{Usually, the generators in bi-complex space and more generally in a multicomplex space are denotes by $i_1,i_2,\hdots,i_n$. Since we are came across $\dl{p}{q}$s from Clifford algebras, we use the symbols $e_1,e_2,\hdots,e_n$ to denote generators.}

We can further generalise to a multicomplex space with $n$ generators. The definition of multicomplex spaces in literature uses the following recursive definition: $\mathbb{M}_n = \{x+ye_n\;|\;x,y\in\mathbb{M}_{n-1}, e_n^2 = -1\}$, $\mathbb{M}_0 = \R$. Note that $\mathbb{M}_1 \cong \dl{0}{1}$ because both of them are exactly the complex numbers. The definition of $\mathbb{M}_n$ is exactly like extending scalars in $\mathbb{M}_1$ from $\R$ to $\mathbb{M}_{n-1}$. Since one can interpret the tensor product as extension of scalars, namely, if $R$ is a subring of $S$, $M$ is an $R$ module, then $S\otimes_R M$ is an $S$ modules containing $M$.
This allows one to interpret $\mathbb{M}_{n-1} \otimes \mathbb{M}_{1}$ as extension of scalars in $\mathbb{M}_1$ from $\R$ to $\mathbb{M}_{n-1}$. Recursively, one can express $\mathbb{M}_{n-1}$ as $\mathbb{M}_{n-2}\otimes\mathbb{M}_{1}$ and so on to get that $\mathbb{M}_n \cong \underbrace{\mathbb{M}_1\otimes\mathbb{M}_1\otimes\cdots\otimes\mathbb{M}_1}_{n \text{ times}}$ and since $\mathbb{M}_1 \cong \dl{0}{1}$, we have that $\mathbb{M}_n \cong \underbrace{\dl{0}{1}\otimes\dl{0}{1}\otimes\cdots\otimes\dl{0}{1}}_{n \text{ times}} \cong \dl{0}{n}$. 

In general, for $\dl{p}{q}$s, the extension of scalars is captured by \eqref{Tensor decomposition Property 1} where $\dl{p}{q}$ is isomorphic to either $\dl{p^\prime}{q^\prime}\otimes\dl{1}{0}$ or $\dl{p^\prime}{q^\prime}\otimes\dl{0}{1}$ where $p^\prime + q^\prime = p+q-1$. This extension of scalars is the intuition behind the proof of Theorem \ref{tensor product decomposition theorem} but first, we first give a more explicit proof for the decomposition.

\subsection{Proofs of the tensor product decomposition}\label{subsection: proof of tensor product decomposition}
\begin{proof}[Proof of Theorem \ref{tensor product decomposition theorem}]The identification:
    $$1 \mapsto 1\otimes 1 \otimes \cdots \otimes 1,$$
    $$e_1 \mapsto {e_1^{a_1}}\otimes 1 \otimes \cdots \otimes 1,$$
    $$e_2 \mapsto 1\otimes {e_1^{a_2}} \otimes \cdots \otimes 1,$$
    $$\vdots$$
    $$e_j \mapsto \underbrace{1\otimes 1 \otimes \cdots \otimes 1}_{j-1 \; \text{times}} \otimes {e_1^{a_j}} \otimes 1 \cdots \otimes 1,$$
    $$\vdots$$
    $$e_n \mapsto 1\otimes 1 \otimes \cdots \otimes {e_1^{a_n}}$$
    is an isomorphism between $\underbrace{\dl{1}{0} \otimes \dl{1}{0} \cdots \otimes \dl{1}{0}}_{p \; \text{times}} \otimes \underbrace{\dl{0}{1} \otimes \dl{0}{1} \cdots \otimes \dl{0}{1}}_{q \; \text{times}}$
    
    \noindent and $\dl{p}{q}$ where $e_1^{a_l}$ denotes the generator of the $l$th factor ($\dl{1}{0}$ or $\dl{0}{1}$ depending on whether the generator squares to $+1$ or $-1$) in the right hand side of the tensor product decomposition.
\end{proof}

We also present another proof using mathematical induction. The aim of giving this other proof is to highlight the `onion like' structure of $\dl{p}{q}$s that was described in Section \ref{Subsection: Example of tensor product decomposition}.
 
\begin{proof}[Another proof of Theorem \ref{tensor product decomposition theorem}]
    We prove \eqref{tensor product decomposition result} using mathematical induction on $n$. Our base case is $n=1$ and there is nothing to prove in the case $n=1$. 
    
    Our induction hypothesis is that the result \eqref{tensor product decomposition result} holds for $n = m$. Assuming this, we now show that if the result holds for $n = m$, then it also holds for $n = m+1$.
    
    We start by observing that a general element of $\dl{p}{q}$, where $p+q = m+1$ can be written as $x + y{e_n}$ where $x,y \in \dl{p^\prime}{q^\prime}$ where $p^\prime + q^\prime = m$ just like what we observed in Section \ref{Subsection: Example of tensor product decomposition}. Now, one can interpret this as extension of scalars from $\R$ to $\dl{p^\prime}{q^\prime}$, thus proving that $\dl{p}{q} \cong \dl{p^\prime}{q^\prime} \otimes \dl{1}{0} \text{ or } \dl{p^\prime}{q^\prime} \otimes \dl{0}{1}$ according to if $e_n^2 = +1 \text{ or } -1$.
    
    Since by our induction hypothesis $\dl{p^\prime}{q^\prime}$ can be decomposed as a tensor product of `$p^\prime$' $\dl{1}{0}$s and `$q^\prime$' $\dl{0}{1}$s, $\dl{p}{q}$ can be decomposed as a tensor product of `$p$' $\dl{1}{0}$s and `$q$' $\dl{0}{1}$s.
\end{proof}

\section{Isomorphisms between different $\dl{p}{q}$}\label{section: isomorphism classes}
The idea of isomorphism classes has been introduced in \cite{MarchukDemo}. We start by recalling it.

Let $A_{p,q}$ be an algebra for all $p,q\in\Z_{\ge0}$ and let $n := p+q$. Then it might be the case that for some fixed $n$, $A_{p,q} \cong A_{p^\prime,q^\prime}$ for some $p,q,p^\prime, q^\prime$ such that $p+q = p^\prime + q^\prime = n$. Therefore, we can divide the collection $\{A_{p,q} \;|\; p+q=n\}$ into different ``classes'' such that members in a class are isomorphic to each other and are not isomorphic to members of a different class.

We give the example of Clifford algebras to make the idea clear.
\subsection{Example: Isomorphism classes in Clifford algebras}
First, consider an example from \cite{MarchukDemo}. Consider Clifford algebras $\cl{p}{q}$ such that $p+q=2$. There are three such algebras:
$$\cl{2}{0},\qquad \cl{1}{1},\qquad \cl{0}{2}.$$
Now, $\cl{2}{0}\cong\cl{1}{1}$. This is because one can think of $e_{12}$ as a generator along with $e_1$ in $\cl{2}{0}$. Since $e_1e_{12} = e_2$, we see that $e_1, e_{12}$ generate all of $\cl{2}{0}$. Also, since $e_{12}^2 = -1$ and $e_1^2=1$, $e_1,e_{12}$ generate a Clifford algebra isomorphic to $\cl{1}{1}$. Therefore, $\cl{2}{0}\cong\cl{1}{1}$. One can check that this technique doesn't work to show that $\cl{2}{0}$ (or $\cl{1}{1}$) $\cong \cl{0}{2}$. One can in fact show that $\cl{2}{0}, \cl{1}{1} \not\cong \cl{0}{2}$. This becomes clear once we consider the isomorphism classes in a general Clifford algebra with $n$ generators using the well-known Cartan--Bott periodicity of 8 theorem:
\begin{thm}[Cartan--Bott periodicity of 8]
    A Clifford algebra $\cl{p}{q}$ is isomorphic to an algebra of matrices as follows: 
    $$\cl{p}{q} \cong
    \left\lbrace
    \begin{array}{ll}
    \Mat(2^{\frac{n}{2}},\R), & \mbox{if $p-q=0, 2 \mod 8$,}\\
    \Mat(2^{\frac{n-1}{2}}, \R)\oplus\Mat(2^{\frac{n-1}{2}}, \R), & \mbox{if $p-q=1 \mod 8$,}\\
    \Mat(2^{\frac{n-1}{2}}, \C), & \mbox{if $p-q=3, 7 \mod 8$,}\\
    \Mat(2^{\frac{n-2}{2}}, \H), & \mbox{if $p-q=4, 6 \mod 8$,}\\
    \Mat(2^{\frac{n-3}{2}}, \H)\oplus\Mat(2^{\frac{n-3}{2}}, \H), & \mbox{if $p-q=5 \mod 8$.}
    \end{array}
    \right.$$
\end{thm}
It is then clear that for $n=p+q$, there are 5 isomorphic classes for Clifford algebras, each isomorphism class corresponds to a algebra of matrices over $\R$, $\R\oplus\R$, $\C$, $\H$ and $\H\oplus\H$.

\subsection{Isomorphism classes in $\dl{p}{q}$}
There are two isomorphism classes for commutative analogues of Clifford algebras which follow from the following theorem.
\begin{thm}\label{thm: isomorphism classes}
    Let $p\in\Z_{\ge0},$ $q\in\N$ and $n :=p+q$. Then:
    \begin{enumerate}
        \item $\dl{n}{0} \not\cong \dl{p}{q}$,
        \item $\dl{p}{q} \cong \dl{0}{n}$.
    \end{enumerate}
\end{thm}

\begin{proof}[Proof of 1]
    Let $U = \sum_{A\subseteq\firstn{n}}{a_A}{e_A} \in \dl{n}{0}$ where $e_A$s are basis elements of $\dl{n}{0}$ and $a_A\in\R$. Then $\proj{U^2}{0} = \sum_{A\subseteq\firstn{n}}{a_A^2} \ge0$. In particular, there is no element in $\dl{n}{0}$ that squares to $-1$. Then it directly follows that $\dl{n}{0} \not\cong \dl{p}{q}$ because $\dl{n}{0} \cong \dl{p}{q}$ would require there to be an element in $\dl{n}{0}$ which squares to $-1$.\qedhere
\end{proof}

\begin{proof}[Proof of 2]
     Let $p\in\Z_{\ge0}$ and $q\in\N, n =p+q$. Define
    \begin{align*}
        f_k &:= {e_k}{e_{p+1}} \text{ for } k \in \{1,2,\ldots,p\},\\
        f_k &:= e_k \text{ for } k \in \{p+1, p+2\ldots,n\}.
    \end{align*}
    Then for all $k\in\firstn{n}$, $f_k^2 = -1$ and $\{f_1,f_2,\ldots,f_n\}$ generate an algebra isomorphic to $\dl{0}{n}$. This gives an isomorphism between $\dl{p}{q}$ and $\dl{0}{n}$. Therefore $\dl{p}{q} \cong \dl{0}{n}$.\qedhere
\end{proof}

The above theorem shows that presence of one generator that squares to $-1$ leads to $n$ generators that square to $-1$. And therefore, there are two isomorphism classes corresponding to presence and absence of a generator that squares to $-1$: if a generator that squares to $-1$ is present, then the isomorphism class is $\dl{0}{n}$, and if there is no generator that squares to $-1$, then the isomorphism class is $\dl{n}{0}$.

\section{Operations of conjugation}\label{Operations of conjugation}
For the rest of the paper, we would require a notion of operations of conjugation in $\dl{p}{q}$. One can think of them as a generalisation of operations of conjugation in complex numbers, split complex numbers, quaternions, etc. Our study of operations of conjugation was inspired from operations of conjugations in Clifford algebras and the general operations of conjugation defined in \cite{OnComputing}. In bicomplex spaces, the 3 operations of conjugations are well known \cite{Price,BCHolFunction,Shapiro}. For tricomplex spaces, operation of conjugation have been discussed in \cite{PlatonicSolids, MSThesis}. In general for multicomplex spaces, the notion of operations of conjugation is fairly new and has been discussed in \cite{ MulticomplexIdeals}.

\subsection{Definition of operations of conjugation}
We define $n = p+q$ operations of conjugation in $\dl{p}{q}$ where each operation of conjugation negates one generator. More explicitly, for $l \in \firstn{n}$, we define $\conj{(.)}{l}: \dl{p}{q} \to \dl{p}{q}$ such that $\conj{U}{l} = U|_{e_l \to -e_l}$. 

\subsection{Examples}\label{operations of conjugation in n = 1}
Consider $\dl{0}{1}$, the complex numbers. Let $U = a + {a_1}{e_1} \in \dl{0}{1}$. The operation of conjugation $\conj{(.)}{1}$ in $\dl{0}{1}$ is the same as complex conjugation denoted by $\overline{(.)}$, $\overline{U} = a - {a_1}{e_1}$. Also, we know from complex number theory that $\overline{UV} = \overline{U} \, \overline{V}.$

Let us consider $\dl{1}{0}$, the split-complex numbers. Let $U = a + {a_1}{e_1} \in \dl{1}{0}$. The operation of conjugation in $\dl{1}{0}$ is same as the split-complex conjugation denoted by $(.)^*$, $\conj{U}{*} = a - {a_1}{e_1}$. Also, we know from split-complex number theory that  $(UV)^* = U^*V^*$.

Let us also consider $\dl{1}{1}$, the commutative quaternions. Let $a + {a_1}{e_1} + {a_2}{e_2} + {a_{12}}{e_{12}} \in \dl{1}{1}$. Then
\begin{align*}
    \conj{U}{1} &= a - {a_1}{e_1} + {a_2}{e_2} - {a_{12}}{e_{12}},\\
    \conj{U}{2} &= a + {a_1}{e_1} - {a_2}{e_2} - {a_{12}}{e_{12}},\\
    \conj{U}{1)(2} &= a - {a_1}{e_1} - {a_2}{e_2} + {a_{12}}{e_{12}}.
\end{align*}
The same operations of conjugations of commutative quaternions have been used in \cite{CommQuatMat}.

\subsection{Properties of operations of conjugation} 
First, we present some first consequences of the definition of operations of conjugations. 
\begin{lem}
\begin{enumerate}
\item The operations of conjugations are involutions, i.e., applying them twice gives the same result as not applying them at all. More explicitly, 
    \begin{equation}\label{properties of operations of conjugations 1}
        \conj{(\conj{U}{l})}{l} = U \text{ for all } U \in \dl{p}{q},\; l \in \firstn{n}.
    \end{equation}
    
\item The operations of conjugation are linear, i.e., 
    \begin{equation}\label{properties of operations of conjugations 2}
        \conj{(aU+bV)}{l} =a\conj{U}{l} + b\conj{V}{l} \text{ for all } U,V \in \dl{p}{q}, a,b\in \R, l \in \firstn{n}.
    \end{equation}
    
\item The operations of conjugations commute with each other, i.e., 
    \begin{equation}\label{properties of operations of conjugations 3}
        \conj{(\conj{U}{l_1})}{l_2} = \conj{(\conj{U}{l_2})}{l_1} \text{ for all } l_1, l_2 \in \firstn{n}.
    \end{equation}
\end{enumerate}
\end{lem}
\smallskip
\begin{proof}
    These properties are obvious, and follow from the definition of operations of conjugation.
\end{proof}

Keeping in mind property 3 \eqref{properties of operations of conjugations 3}, we denote superposition of many operations of conjugations by just writing them next to each other, for example, $\conj{(\conj{(\conj{U}{1})}{2})}{3} = \conj{U}{1)(2)(3}$. We introduce the following notation for a superposition of operations of conjugation: let $A = \{i_1,i_2,\hdots, i_k\} \subseteq \firstn{n}$, then we define for $U \in \dl{p}{q}$
\begin{equation}\label{notation for superposition of operations of conjugation}
    \conj{U}{A} := \conj{U}{i_1)(i_2)\ldots(i_k}.
\end{equation}
We also define $\conj{U}{\{\}} := U$. Let $A\subseteq\firstn{n}$. We call $\conj{U}{A}$ a conjugate of $U$. 

\begin{rem}\label{operations of conjugations form an abelian group}
    Note that the set of all superpositions of operations of conjugates form an Abelian group (as mentioned in \cite{MulticomplexIdeals} for the particular case of multicomplex spaces). The group operation in the group is superposition of operations of conjugation and the identity element of this group is the operation of conjugation $\conj{(.)}{\{\}}$. It is easy to see that this group is isomorphic to $\left(\frac{\Z}{2\Z}\right)^n$.
\end{rem}

Next, we have a non-trivial result -- operations of conjugation distribute over multiplication. In Clifford algebras, only the grade involution distributes over multiplication. Grade reversion and Clifford conjugation reverse the order of products when applied to a product.
\begin{thm}
Operations of conjugation distribute over multiplication.
\begin{equation}\label{properties of operations of conjugations 4}
    \conj{(UV)}{l} = \conj{U}{l}\conj{V}{l} \text{ for all } U,V \in \dl{p}{q}, l \in \firstn{n}.
\end{equation}
\end{thm}

\begin{proof}
    We prove this using mathematical induction on $n:=p+q$. We have proved the result for n = 1 in Section \ref{operations of conjugation in n = 1}. This will be our base case.
    
    Our induction hypothesis is that \eqref{properties of operations of conjugations 4} holds for $n = m$, i.e., $\conj{(UV)}{l} = \conj{U}{l}\conj{V}{l}$ for all $U,V \in \dl{p}{q},\; l\in\firstn{m}, p+q = m$. We will now show that \eqref{properties of operations of conjugations 4} holds for $n=m+1$.
    
    Let $p+q=m+1$. Let $U,V \in \dl{p}{q}$. Then by \eqref{Tensor decomposition Property 1}, $\dl{p}{q} \cong \dl{p^\prime}{q^\prime} \otimes \dl{1}{0} \text{ or } \dl{p^\prime}{q^\prime} \otimes \dl{0}{1}$, according to if $e_{m+1}^2 = +1 \text{ or } -1$, where $p^\prime + q^\prime = m$.  Using this isomorphism, we can express $U, V$ as $U = {x_1} + {y_1}{e_{m+1}}$, $V = {x_2} + {y_2}{e_{m+1}}$ where ${x_1}, {x_2}, {y_1}, {y_2} \in \dl{p^\prime}{q^\prime}$. Now, $UV = {x_1}{x_2} + e_{m+1}^2{y_1}{y_2} + ({x_1}{y_2} + {x_2}{y_1})e_{m+1}$. Using the linearity of operations of conjugation, we get
    \begin{align*}
        \conj{(UV)}{l} &= \conj{({x_1}{x_2} + e_{m+1}^2{y_1}{y_2} + ({x_1}{y_2} + {x_2}{y_1})e_{m+1})}{l}\\
        &= \conj{({x_1}{x_2})}{l} + e_{m+1}^2\conj{({y_1}{y_2})}{l} + \conj{({x_1}{y_2}e_{m+1})}{l} + \conj{({x_2}{y_1}e_{m+1})}{l}.
    \end{align*}
    If $l \le m$, we use the induction hypothesis to get
    \begin{align*}
        \conj{(UV)}{l} &= \conj{({x_1}{x_2})}{l} + e_{m+1}^2\conj{({y_1}{y_2})}{l} + \conj{({x_1}{y_2}e_{m+1})}{l} + \conj{({x_2}{y_1}e_{m+1})}{l}\\
        &= \conj{x_1}{l}\conj{x_2}{l} + e_{m+1}^2\conj{y_1}{l}\conj{y_2}{l} + e_{m+1}(\conj{x_1}{l}\conj{y_2}{l} + \conj{x_2}{l}\conj{y_1}{l})\\
        &= (\conj{x_1}{l} + \conj{y_1}{l}e_{m+1})(\conj{x_2}{l} + \conj{y_2}{l}e_{m+1}) = \conj{U}{l}\conj{V}{l}.
    \end{align*} 
    Thus, $\conj{(UV)}{l} = \conj{U}{l}\conj{V}{l}$ if $l \le m$.
    
    If $l = m+1$, then $\conj{\alpha}{l} = \alpha$ for $\alpha \in \{x_1, y_1, x_2, y_2\}$. We get 
   \begin{align*}
        \conj{(UV)}{m+1} &= \conj{({x_1}{x_2})}{m+1} + e_{m+1}^2\conj{({y_1}{y_2})}{m+1} +\\
        & \quad \conj{({x_1}{y_2}e_{m+1})}{m+1} + \conj{({x_2}{y_1}e_{m+1})}{m+1}\\
        &= {x_1}{x_2} + e_{m+1}^2{y_1}{y_2} -e_{m+1}({x_1}{y_2} + {x_2}{y_1})\\ 
        &= ({x_1} - {y_1}e_{m+1})({x_2} - {y_2}e_{m+1}) = \conj{U}{m+1}\conj{V}{m+1}.
   \end{align*}
    Thus, $\conj{(UV)}{l} = \conj{U}{l}\conj{V}{l}$ for all $l \in \firstn{n}$.
\end{proof}

\section{Special basis in $\dl{p}{q}$}\label{Section: Special basis in K(p,q)}
It follows from Section \ref{section: isomorphism classes} that a general commutative analogue of Clifford algebras $\dl{p}{q}$ is isomorphic to either $\dl{n}{0}$ or $\dl{0}{n}$. Therefore, without loss of generality, we will work only with $\dl{n}{0}$ and $\dl{0}{n}$ in this section. Previously, we have considered the tensor product of algebras. Now, we turn our attention to direct sum of algebras. First, we quickly recall the definition of direct sum of algebras. 

\subsection{Direct sums of algebras}
Let $A_1,A_2,\hdots,A_n$ be $n$ finite-dimensional associative algebras with unity. Since each $A_k$ is an algebra, it is also a vector space and we can consider the direct sum of $A_k$s thinking of them as vector spaces: $A_1\oplus A_2\oplus \cdots \oplus A_n$. Now, we can endow this direct sum vector space with a `natural' multiplication and make $A_1\oplus A_2\oplus \cdots \oplus A_n$ into an algebra. Let $(a_1,a_2,\hdots,a_n), (b_1,b_2,\hdots,b_n) \in A_1\oplus A_2\oplus \cdots \oplus A_n$. We define multiplication coordinate-wise
$$(a_1,a_2,\hdots,a_n) \cdot (b_1,b_2,\hdots,b_n) := ({a_1}\cdot{b_1},{a_2}\cdot{b_2},\hdots,{a_n}\cdot{b_n}).$$
This multiplication makes $A_1\oplus A_2\oplus \cdots \oplus A_n$ into a finite-dimensional associative algebra with unity. Let $1_1, 1_2, \hdots, 1_n$ be the multiplicative identities of $A_1, A_2, \hdots, A_n$ respectively. Then the multiplicative identity of $A_1\oplus A_2\oplus \cdots \oplus A_n$ is $(1_1, 1_2, \ldots, 1_n)$. We discuss important examples: direct sums of $\R$s and $\C$s. This will be helpful in the rest of the paper.

\subsection{Direct sums of $\R$s and $\C$s}\label{subsection: direct sums of Rs and Cs}
$\R$ and $\C$ are well-known algebras. Let us consider their direct sums 
\begin{eqnarray}
\R^m := \underbrace{\R\oplus\R\oplus\cdots\oplus\R}_{m \,\, \text{times}},\qquad \C^m := \underbrace{\C\oplus\C\oplus\cdots\oplus\C}_{m \,\, \text{times}},
\end{eqnarray}
where the multiplication in them are defined coordinate wise just like we discussed above. Let us take a look at multiplication of basis elements in them. The reason for discussing multiplication of basis elements in the algebras $\R^m$ and $\C^m$ is because we will show that $\dl{n}{0}$ is isomorphic to a direct sum of $\R$s and $\dl{0}{n}$ is isomorphic to a direct sum of $\C$s.

\subsubsection{Direct sum of $\R$s}
For each $k\in\firstn{m}$ let $b_k=(0,\hdots,1,\hdots,0) \in\R^m$ be such that there is 1 at the $k$\textsuperscript{th} position and $0$ elsewhere. Then $\{b_1,b_2,\hdots,b_m\}$ is the standard basis for $\R^m$. The multiplication of basis elements is quite simple:
\begin{equation}\label{EQn: multiplication in direct sums of reals}
    b_k b_l = \delta_{kl}b_k
\end{equation}
where $\delta_{kl}$ is the Kronecker delta\footnote{$\delta_{kl} = 1$ if $k=l$ and $\delta_{kl} = 0$ otherwise.}. Writing explicitly, $b_k b_l = (0,0,\hdots,0)$ if $k\ne l$ and $b_k^2 = b_k$ i.e., $b_k$s are idempotents. A consequence of this structure in $\R^m$ is that if we have an algebra consisting of basis elements whose multiplication resembles \eqref{EQn: multiplication in direct sums of reals}, then that algebra is isomorphic to a direct sum of $\R$s.

Now, we investigate direct sums of $\C$s.
\subsubsection{Direct sum of $\C$s}
We will consider $\C^m$ as an real algebra, i.e., we will take the scalars to be real numbers. Therefore, $\C^m$ is $2m$-dimensional real algebra.

For each $k\in\firstn{m}$ let $b_k^1 :=(0,\hdots,1,\hdots,0) \in\R^m$ be such that there is 1 at the $k$\textsuperscript{th} position and $0$ elsewhere. Let $b_k^i := ib_k^1 = (0,\hdots,i,\hdots,0)\in\C^m$. The multiplication of basis elements is as follows:
\begin{eqnarray}\label{EQn: multiplication in direct sums of complex numbers}
    b_k^1 b_l^1 = \delta_{kl} b_k^1,\qquad\qquad
    b_k^i b_l^i = -\delta_{kl} b_k^1,
\end{eqnarray}
that is, basis elements with different subscripts multiply to $0$, $b_k^1$s are idempotents and $b_k^i$s squares to $-b_k^1$. A consequence of this structure in $\C^m$ is that if we have a real algebra consisting of basis elements whose multiplication resembles \eqref{EQn: multiplication in direct sums of complex numbers}, then that algebra is isomorphic to a direct sum of $\C$s.

Now we come back to the algebras $\dl{n}{0}$ and $\dl{0}{n}$. What we seek in the next section is a basis in $\dl{n}{0}$ such that the multiplication of its elements resembles \eqref{EQn: multiplication in direct sums of reals} and a basis in $\dl{0}{n}$ such that its multiplication resembles \eqref{EQn: multiplication in direct sums of complex numbers}. We call these bases `special'. We start the hunt for these bases with a motivation, the reason why we expect them to exist.

\subsection{Special basis for $\dl{n}{0}$}\label{subsection: special basis for K(n,0)}
\subsubsection{Motivation}
It is well known that the split complex numbers $\dl{1}{0}$ are isomorphic to $\R\oplus\R$. Consider the elements $f_+ = \frac{1 + {e_1}}{2}$ and $f_- = \frac{1 - {e_1}}{2}$ in $\dl{1}{0}$. Then it follows that 
\begin{align*}
    f_+^2 &= f_+,\\
    f_-^2 &= f_-,\\
    {f_+}{f_-} &= 0.
\end{align*}
As discussed in Section \ref{subsection: direct sums of Rs and Cs} that is then clear the identification $f_+ \leftrightarrow (1,0), f_- \leftrightarrow (0,1)$ is an isomorphism between $\dl{1}{0}$ and $\R\oplus\R$.

Now, since $\dl{n}{0} \cong \underbrace{\dl{1}{0} \otimes \dl{1}{0} \cdots \otimes \dl{1}{0}}_{n \; \text{times}}$, using the distributivity of tensor product over direct sums and $\R\otimes\R$ being isomorphic to $ \R$, it follows that in general,
$$\dl{n}{0} \cong \underbrace{\R\oplus\R\oplus\cdots\oplus\R}_{2^n \text{ times}} = \R^{2^n}.$$
We will give the isomorphism between $\dl{n}{0}$ and $\R^{2^n}$ next, but before that, note that if we define $f = \frac{1}{2}(1 + {e_1}) \in \dl{1}{0}$, then $f_+ = \conj{f}{\{\}}$ and $f_- = \conj{f}{1}$. We will see that a generalization of this will result in an isomorphism between $\dl{n}{0}$ and $\underbrace{\R\oplus\R\oplus\cdots\oplus\R}_{2^n \text{ times}}$.

\subsubsection{The basis}
Let $f \in \dl{n}{0}$ be the sum of all basis elements in $\dl{n}{0}$ modulo a constant i.e.,
$$f := \frac{1}{2^n}\left(\sum_{A\subseteq\firstn{n}}{e_A}\right).$$
For $B\subseteq\firstn{n}$, define 
$$f_B := \conj{f}{B}.$$
The pre-factor $\frac{1}{2^n}$ can be thought of as a normalisation constant. We give some interesting results which follow from $f$ being a very symmetric combination of generators $e_1,e_2,\hdots,e_n$.

\begin{lem}\label{Lemma: properties of f_Bs} For $B\subseteq\firstn{n}$, we have\\

    1. $f_B^2=f_B$,\\

    2. ${e_k}{f_B} = \left\{
        \begin{aligned}
            &f_B,& \mbox{if $k \notin B$,}\\
            &-f_B,& \mbox{if $k \in B$.}
        \end{aligned}\right.$\\
        
    3. $\sum_{B\subseteq\firstn{n}} f_B = 1.$
\end{lem}

\begin{proof}[Proof of 1]
    First, we will prove using mathematical induction on $n$, that $f^2=f$.

    We first show that the result holds for $n=1$. When $n=1$, $f = \frac{1}{2}(1+e_1)$. Then 
    \begin{align*}
        f^2 &= \frac{1}{2}(1+e_1)\frac{1}{2}(1+e_1) = \frac{1}{4}(2 + 2{e_1}) = f.
    \end{align*} 
    This is our base case.

    Our induction hypothesis is that the result holds for $n = m-1$. Assuming this, we will prove that the result holds for $n=m$.
    Let 
    $$x = \frac{1}{2^{m-1}}\left(\sum_{A\subseteq\{1,2,\hdots,m-1\}}{e_A}\right).$$
    Then $f = \frac{1}{2}(x+{e_m}x)$. Now, 
    \begin{align*}
        \!f^2 = \frac{1}{2}\left({x+{e_m}x}\right)\frac{1}{2}\left({x+{e_m}x}\right) = \frac{1}{4}(x^2 + {e_m}x + x{e_m} + x^2) = \frac{1}{2}(x^2 + {e_m}x).
    \end{align*}
    By our induction hypothesis, $x^2 = x$. Therefore,
    \begin{align*}
        f^2 &= \frac{1}{2}(x+{e_m}x) = f.
    \end{align*}
    Let $B\subseteq\firstn{n}$. Since $f^2=f$, applying operation of conjugation $\conj{(.)}{B}$ on both sides, we get $f_B^2=f_B$. This completes the proof.
\end{proof}

\begin{proof}[Proof of 2]
    We will first show that for all $k\in\firstn{n}$, $e_k f = f$. Let 
    $$x = \frac{1}{2^{n-1}}\left(\sum_{A\subseteq\firstn{n}\setminus\{k\}}{e_A}\right).$$
    Then $f = \frac{1}{2}(x+{e_k}x)$. Now, 
    \begin{align*}
        {e_k}f &= {e_k}(\frac{1}{2}(x+{e_k}x)) = \frac{1}{2}({e_k}x+x) = f.
    \end{align*}
    Let $B\subseteq\firstn{n}$. Applying operation of conjugation $\conj{(.)}{B}$ on both sides, we get $\conj{({e_k}f)}{B} = \conj{f}{B}$ which implies $\conj{e_k}{B}\conj{f}{B} = \conj{f}{B}$. Now, since $\conj{e_k}{B} = \left\{
    \begin{aligned}
            &e_k,& \mbox{if $k \notin B$,}\\
            &-e_k,& \mbox{if $k \in B$,}
\end{aligned}\right.$ the result follows.
\end{proof}

\begin{proof}[Proof of 3]
    Since $f_B = \conj{f}{B}$, when we add all $f_B$s together, all the terms in which a generator appears cancel out because in the sum, there are an equal number of terms involving $e_A$s with a plus sign and terms involving $e_A$s with a minus sign. Therefore, in the sum $\sum_{B\subseteq\firstn{n}} f_B$, only the grade $0$ terms survive and their sum is equal to 1 because each $f_B$ contributes $\frac{1}{2^n}$ and there are $2^n$ $f_B$s. 
\end{proof}
 
\begin{lem}\label{Lemma: orthogonality of f_Bs}
    Let $B$ be a non-empty subset of $\firstn{n}$, then $f{f_B} = 0$.
\end{lem}
\begin{proof}
    Let $B\subseteq\firstn{n}$ such that $|B|>0$. Then we can express $B$ as 
    $$B = B^\prime\sqcup\{k\}$$
    where $k\in B$ and $B^\prime = B\setminus\{k\}$. Let 
    $$x = \frac{1}{2^{n-1}}\left(\sum_{A\subseteq\firstn{n}\setminus\{k\}}{e_A}\right).$$
    Then $f = \frac{1}{2}(x+{e_k}x)$. Consider $\conj{f}{k}\conj{f}{B^\prime}$:
    \begin{multline*}
        \conj{f}{k}\conj{f}{B^\prime} = \frac{1}{2}(x-{e_k}x)\frac{1}{2}\conj{(x+{e_k}x)}{B^\prime}\\ 
        = \frac{1}{4}(x\conj{x}{B^\prime} + {e_k}x\conj{x}{B^\prime} - {e_k}x\conj{x}{B^\prime} - x\conj{x}{B^\prime}) = 0.
    \end{multline*}
    Now, 
    \begin{align*}
        f{f_B} &= f\conj{f}{B} = f\conj{f}{B^\prime)(k} = \conj{(\conj{f}{k}\conj{f}{B^\prime})}{k} = \conj{0}{k} =0.\qedhere
    \end{align*}
\end{proof}

Let $U\in\dl{n}{0}$. Then for each $A \subseteq\firstn{n}$, define linear map $\zeta_A:\dl{n}{0}\to\R$ by defining $\zeta_A(e_C) := (-1)^{|A\cap C|}$, i.e., we replace generators in $e_C$ by $-1$ if their index is present in $A$ and $+1$ if their index is absent in $A$. We have defined $\zeta_A$ on the basis elements and one then extends it to all of $\dl{n}{0}$ via linearity.
\begin{lem}\label{Theorem: real irred reps are coefficients}
    Let $U\in\dl{n}{0}$, $B\subseteq\firstn{n}$. Then $Uf_B = {\zeta_B}(U)f_B$.
\end{lem}
\begin{proof}
    Let $B\subseteq\firstn{n}$. Let $A = \{a_1,a_2,\hdots,a_m\} \subseteq\firstn{n}$. We will find ${e_A}{f_B}$. Recall from Lemma \ref{Lemma: properties of f_Bs} that $${e_k}{f_B} = \left\{
        \begin{aligned}
            &f_B,& \mbox{if $k \notin B$},\\
            &-f_B,& \mbox{if $k \in B$.}
        \end{aligned}\right.$$
    Iteratively using this, we get,
    $${e_A}{f_B} = (-1)^{|A\cap B|}{f_B}.$$
    Let $U = \sum_{A\subseteq\firstn{n}}{{a_A}{e_A}} \in \dl{n}{0}$, $B\subseteq\firstn{n}$. Then 
    \begin{align*} 
        U{f_B} &= \sum_{A\subseteq\firstn{n}}{{a_A}({e_A}{f_B})} = \sum_{A\subseteq\firstn{n}}{{a_A}((-1)^{|A\cap B|}{f_B})}\\
        &= \sum_{A\subseteq\firstn{n}}{((-1)^{|A\cap B|}{a_A}){f_B}} = \left(\sum_{A\subseteq\firstn{n}}{(-1)^{|A\cap B|}{a_A}}\right){f_B}\\
        &= {\zeta_B}(U){f_B}
    \end{align*}
    because $\sum_{A\subseteq\firstn{n}}{(-1)^{|A\cap B|}{a_A}} = {\zeta_B}(U)$ since the coefficient $a_A$ gets negated for each element of $A$ in $B$.
\end{proof}

\begin{thm}\label{Theorem: f_Bs is orthonormal basis}
    Let $A,B \subseteq\firstn{n}$, then $${f_A}{f_B} = \left\{
        \begin{aligned}
            &f_A,& \mbox{if $A=B$},\\
            &0,& \mbox{if $A\ne B$.}
        \end{aligned}\right.$$ Further, $\{f_B \;|\; B\subseteq\firstn{n}\}$ is a basis for $\dl{n}{0}$.
\end{thm}
\begin{proof}
    Let $A\subseteq\firstn{n}$. From Lemma \ref{Lemma: properties of f_Bs}, it follows that ${f_A}^2 = {f_A}$. Now, if $A,B\subseteq\firstn{n}$ such that $A\ne B$, then note that 
    $${f_A}{f_B} = \conj{f}{A}\conj{f}{B} = \conj{(f\conj{f}{A)(B})}{A} = \conj{(f\conj{f}{A\va B})}{A},$$ where $A\va B$ denotes symmetric difference of the sets $A$ and $B$, i.e., $A\va B = A\cup B\setminus A\cap B$. Since $A\ne B$, $A\va B \ne \{\}$. Then by using Lemma \ref{Lemma: orthogonality of f_Bs}, we get $f\conj{f}{A\va B} = 0$ and therefore, ${f_A}{f_B} = 0$.

    Now, we will show that the set $\{f_B \;|\; B\subseteq\firstn{n}\}$ is a basis for $\dl{n}{0}$. First, we show linear independence. Let $d_B\in\R$ be such that $\sum_{B\subseteq\firstn{n}}{{d_B}{f_B}} = 0$. Multiplying both sides by $f_A$ we get
    $$\sum_{B\subseteq\firstn{n}}{{d_B}{f_B}{f_A}} = 0.$$
    
    Since ${f_A}{f_B} = \left\{
        \begin{aligned}
            &f_A,& \mbox{if $A=B$},\\
            &0,& \mbox{if $A\ne B$,}
        \end{aligned}\right.$ we get 
        $$\sum_{B\subseteq\firstn{n}}{{d_B}{f_B}{f_A}} = {d_A}{f_A} = 0.$$
        
    Since $f_A \ne 0$, $d_A = 0$. Since $A\subseteq\firstn{n}$ was arbitrary, we get $d_A = 0$ for all $A\subseteq\firstn{n}$. This shows linear independence.

    Lastly, we show that $\{f_B \;|\; B\subseteq\firstn{n}\}$ spans $\dl{n}{0}$. Let $U\in\dl{n}{0}$. Then using Lemma \ref{Theorem: real irred reps are coefficients} and result 3 of Lemma \ref{Lemma: properties of f_Bs}, we have that
    \begin{multline}\label{basis change formula for K_{n,0}}
        U = U\cdot1 = U\cdot\left(\sum_{B\subseteq\firstn{n}} f_B\right)\\ = \sum_{B\subseteq\firstn{n}}Uf_B =  \sum_{B\subseteq\firstn{n}}{({\zeta_B}(U)){f_B}}.
    \end{multline}
    
    Therefore, the set $\{f_B \;|\; B\subseteq\firstn{n}\}$ spans $\dl{n}{0}$.
\end{proof}

We need an enumeration of the set of subsets of $\firstn{n}$. We consider the enumeration given in the paper \cite{A.Acus}. We arrange the subsets of $\firstn{n}$ in lexicographic order and then enumerate them. Let us denote the index of a subset $B\subseteq\firstn{n}$ according to this enumeration by $i(B)$. We index the empty set $\{\}$ as the first set.

Let $b_k = (0,0,\hdots,0,1,0,\hdots,0) \in \R^{2^n}$ be the tuple having zeros in all positions except the $k$\textsuperscript{th} position as defined previously.
\begin{thm}\label{Theorem: isomorphism between K_{n,0} and direct sum of Rs}
    The identification $f_B \leftrightarrow b_{i(B)}$ is an isomorphism between $\dl{n}{0}$ and $\R^{2^n}$.
\end{thm}
\begin{proof}
    The proof is a direct consequence of Theorem \ref{Theorem: f_Bs is orthonormal basis} and the discussion given in Section \ref{subsection: direct sums of Rs and Cs}.
\end{proof}

\subsection{Special basis for $\dl{0}{n}$}\label{subsection: special basis in K_{0,n}}
\subsubsection{Motivation}
It turns out that $\C\otimes\C\cong \C\oplus\C$. The isomorphism is given by ${z_1}\otimes{z_2} \mapsto ({z_1}{z_2},\;{z_1}\overline{z_2})$.\footnote{The author first came across this in this Physics forums post: \href{https://www.physicsforums.com/threads/mathbb-c-oplus-mathbb-c-cong-mathbb-c-otimes-mathbb-c.1055636}{https://www.physicsforums.com/threads/mathbb-c-oplus-mathbb-c-cong-mathbb-c-otimes-mathbb-c.1055636}.} Since $\dl{0}{1} \cong \C$, this means that $\dl{0}{2}\cong\C\oplus\C$. Writing the isomorphism explicitly, we get
\begin{align*}
    1\otimes1 &\leftrightarrow 1 \mapsto (1,1),\\
    {e_1}\otimes1 &\leftrightarrow e_1 \mapsto (i,i),\\
    1\otimes{e_1} &\leftrightarrow e_2 \mapsto (i,-i),\\
    {e_1}\otimes{e_1} &\leftrightarrow e_{12} \mapsto (-1,1).
\end{align*}

Equivalently, we can write the above isomorphism as\footnote{Here 'e' means even and 'o' means odd. Specifically 'e' is used in the superscript of $f$ when $f$ has elements of even grade and 'o' is used in the superscript of $f$ when $f$ has elements of odd grade. We use this notation because it inspires a generalisation to $\dl{0}{n}$ which we present in the Section \ref{subsubsection: special basis in K_{0,n}}.}
\begin{align*}
    f_1^e &= \frac{1}{2}\left(1 - e_{12}\right) \mapsto (1,0),\\
    f_1^o &= \frac{1}{2}\left({e_1}+{e_2}\right) \mapsto (i,0),\\
    f_2^e &= \frac{1}{2}\left({e_1}-{e_2}\right) \mapsto (0,i),\\
    f_2^o &= \frac{1}{2}\left(1 + e_{12}\right) \mapsto (0,1).
\end{align*}

Now, since $\dl{0}{n} \cong \underbrace{\dl{0}{1} \otimes \dl{0}{1} \otimes\cdots \otimes \dl{0}{1}}_{n \; \text{times}}$, using the distributivity of tensor product over direct sums and $\C\otimes\C$ being isomorphic to $\C\oplus\C$, it follows that in general,
$$\dl{0}{n} \cong \underbrace{\C\oplus\C\oplus\cdots\oplus\C}_{2^{n-1} \text{ times}} = \C^{2^{n-1}}.$$

We will give the isomorphism between $\dl{0}{n}$ and $\C^{2^{n-1}}$ next. But before that, note that if we define $f := \frac{1}{2}\left(1 + {e_1} + {e_2} + e_{12}\right), f^e = \proj{f}{0} - \proj{f}{2} = \frac{1}{2}\left(1-e_{12}\right)$ and $f^o = \proj{f}{1} = \frac{1}{2}\left(e_1 + e_2\right) \in \dl{0}{2}$, then $f_1^e = \conj{(f^e)}{\{\}}, f_1^o = \conj{(f^o)}{\{\}}, f_2^e = \conj{(f^e)}{2}, f_2^o = \conj{(f^o)}{2}$.
We will see that a generalization of this will result in an isomorphism between $\dl{0}{n}$ and $\C^{2^{n-1}}$.

\subsubsection{The basis}\label{subsubsection: special basis in K_{0,n}}
Let $f \in \dl{0}{n}$ be the sum of all basis elements in $\dl{0}{n}$ modulo a constant  i.e.,
$$f := \frac{1}{2^{n-1}}\left(\sum_{A\subseteq\firstn{n}}{e_A}\right).$$
The pre-factor $\frac{1}{2^{n-1}}$ can be thought of as a normalisation constant.
Define\footnotemark 
\begin{align*}
    E &:= \,\sum_{k=0}^{\lfloor\frac{n}{2}\rfloor}(-1)^{k}\;\proj{f}{2k} = \proj{f}{0}-\proj{f}{2}+\proj{f}{4}-\cdots,\\
    O &:= \;\sum_{k=0}^{\lfloor\frac{n}{2}\rfloor}(-1)^{k}\;\proj{f}{2k+1} = \proj{f}{1}-\proj{f}{3}+\proj{f}{5}-\cdots.
\end{align*}
\footnotetext{The element $E$ is the element $\Gamma_n = \gamma_2\gamma_3\cdots\gamma_n$ presented in \cite{MulticomplexIdeals} but in general our approach to idempotents differs from the approach presented in \cite{MulticomplexIdeals}.}

For $B\subseteq\{2,3,\hdots,n\}$, define $O_B := \conj{O}{B}, E_B := \conj{E}{B}$. We give some interesting results which follow from $O$ and $E$ being a very symmetric combination of generators $e_1,e_2,\hdots,e_n$.
\begin{rem}\label{remark about convention of E_Bs}
    The reason why we have defined $E_B$ for $B\subseteq\{2,3,\hdots,n\}$ and not for $B\subseteq\firstn{n}$ is because for all $A\subseteq\firstn{n}$, we have $E_{A} = E_{A^\mathsf{c}}$ where $A^\mathsf{c} = \firstn{n} \setminus A$ is the complement of the set $A$. This follows because $E_{\firstn{n}} = E$ holds since $E$ consists of basis elements of even grades and negating all generators in an even grade element will not change it. Then for all $A\subseteq\firstn{n}$, $\firstn{n} = A \sqcup A^{\mathsf{c}}$ and thus $E = E_{\firstn{n}} = E_{A \sqcup A^{\mathsf{c}}} = \conj{E}{A) (A^{\mathsf{c}}}$ holds and applying operation of conjugation $\conj{(.)}{A}$ on both sides gives us $E_A = E_{A^\mathsf{c}}$.
    
    Therefore it is redundant to consider all of the elements $\{E_A \;|\; A\subseteq\firstn{n}\}$ because this set contains duplicates -- half of the elements are the same.
    
    What one can do is to remove one element from $\firstn{n}$, say $k$. Now a subset $A$ of $\firstn{n}$ either contains $k$ or it does not contain $k$. Further, if $k\in A$ then $k \not\in A^\mathsf{c}$ and if $k \in A^\mathsf{c}$ then $k \not\in A$. Therefore, one can divide the set of all subsets of $\firstn{n}$ into two classes: collection of subsets that contain $k$ and the collection of subsets that do not contain $k$. Since $E_A = E_{A^\mathsf{c}}$, one can then only consider $E_A$ where $A$ lies in any of the two classes. In this paper, we choose $k=1$ and consider the subsets that do not contain $1$ i.e., the subsets of the set $\{2,3,\hdots,n\}$. One could choose any arbitrary element of $\firstn{n}$ but we in this paper decided to choose $k=1$ and therefore all the results presented henceforth are for the elements $\{E_B\;|\;B\subseteq\{2,3,\hdots,n\}\}$.
\end{rem}
\begin{lem}\label{Lemma: first properties of special basis in K(0,n)}
    For all $B\subseteq\{2,3,\hdots,n\}$, we have\\
    
    1. $O_B^2 = - E_B, E_B^2 = E_B$ and ${E_B}{O_B} = O_B$.\\
    
    2. For all $k\in\{2,3,\hdots,n\}$,
    \begin{align*}
        {e_k}{O_B} = \left\{
        \begin{aligned}
            &E_B,& \mbox{if $k\in B$,}\\
            &-E_B,& \mbox{if $k\notin B,$}
        \end{aligned}
        \right.\quad
        {e_k}{E_B} = \left\{
        \begin{aligned}
            &-O_B,& \mbox{if $k\in B$,}\\
            &O_B,& \mbox{if $k\notin B.$}
        \end{aligned}
        \right.
    \end{align*}
    
    3. $\sum_{B\subseteq\{2,3,\hdots,n\}} E_B = 1$.
\end{lem}
\begin{proof}[Proof of 1]
        We will first show that $O^2 = -E, E^2 = E$ and $EO=O$. We will do this by using mathematical induction on $n$. Consider the case when $n=2$. Then $f=\frac{1}{2}\left(1+e_1+e_2+e_{12}\right), O=\frac{1}{2}\left(e_1+e_2\right), E=\frac{1}{2}\left(1-e_{12}\right)$. It is easily checked that $O^2 = -E, E^2 = E$ and $EO=O$. This is our base case.

        Our induction hypothesis is that the result holds for $n=m-1$. Assuming this, we will prove that the result holds for $n=m$. Let 
        \begin{align*}
            f^\prime &= \frac{1}{2^{m-2}}\left(\sum_{A\subseteq\{1,2,\hdots,m-1\}}\right),\\
            E^\prime &= \,\sum_{k=0}^{\lfloor\frac{m-1}{2}\rfloor}(-1)^{k}\;\proj{f^\prime}{2k} = \proj{f^\prime}{0}-\proj{f^\prime}{2}+\proj{f^\prime}{4}-\cdots,\\
            O^\prime &= \;\sum_{k=0}^{\lfloor\frac{m-1}{2}\rfloor}(-1)^{k}\;\proj{f^\prime}{2k+1} = \proj{f^\prime}{1}-\proj{f^\prime}{3}+\proj{f^\prime}{5}-\cdots.
        \end{align*}
        
        Note that then 
        \begin{align*}
            f &= \frac{1}{2}\left(f^\prime + {e_m}f^\prime\right),\hspace{10pt} E = \frac{1}{2}\left(E^\prime - {e_m}O^\prime\right),\hspace{10pt} O = \frac{1}{2}\left(O^\prime + {e_m}E^\prime\right).
        \end{align*}
        
        Now,
        \begin{align*}
            E^2 &= \frac{1}{2}\left(E^\prime - {e_m}O^\prime\right)\frac{1}{2}\left(E^\prime - {e_m}O^\prime\right) = \frac{1}{4}(E^\prime - {e_m}O^\prime)(E^\prime - {e_m}O^\prime)\\
            &= \frac{1}{4}\left((E^\prime)^2 - {e_m}{E^\prime}{O^\prime} - {e_m}{O^\prime}{E^\prime} - (O^\prime)^2\right).
        \end{align*}
        
        Since by our induction hypothesis, we have that $(E^\prime)^2 = {E^\prime}, (O^\prime)^2 = -{E^\prime}$ and ${E^\prime}{O^\prime}={O^\prime}$, we get
        \begin{align*}
            E^2 &= \frac{1}{4}(2{E^\prime} -2{e_m}{O^\prime}) = \frac{1}{2}\left(E^\prime - {e_m}O^\prime\right) = E.
        \end{align*}
        
        Similarly, 
        \begin{align*}
            O^2 &= \frac{1}{2}\left(O^\prime + {e_m}E^\prime\right)\frac{1}{2}\left(O^\prime + {e_m}E^\prime\right) = \frac{1}{4}(O^\prime + {e_m}E^\prime)(O^\prime + {e_m}E^\prime)\\
            &= \frac{1}{4}((O^\prime)^2 + {e_m}{E^\prime}{O^\prime} + {e_m}{O^\prime}{E^\prime} - (E^\prime)^2).
        \end{align*}
        
        Again, using our induction hypothesis, we have that $(E^\prime)^2 = {E^\prime}, (O^\prime)^2 = -{E^\prime}$ and ${E^\prime}{O^\prime}={O^\prime}$, and we get
        \begin{align*}
            O^2 &= \frac{1}{4}(-2{E^\prime} + 2{e_m}{O^\prime}) = \frac{1}{2}(-{E^\prime} + {e_m}{O^\prime}) = -\frac{1}{2}({E^\prime} - {e_m}{O^\prime}) = -E.
        \end{align*}
        
        Similarly, once again,
        \begin{align*}
            EO &= \frac{1}{2}\left(E^\prime - {e_m}O^\prime\right)\frac{1}{2}\left(O^\prime + {e_m}E^\prime\right) = \frac{1}{4}(E^\prime - {e_m}O^\prime)(O^\prime + {e_m}E^\prime)\\
            &= \frac{1}{4}({E^\prime}{O^\prime} + {e_m}(E^\prime)^2 -{e_m}(O^\prime)^2 + {O^\prime}{E^\prime}).
        \end{align*}
        
        And using our induction hypothesis for the last time, we have that $(E^\prime)^2 = {E^\prime}, (O^\prime)^2 = -{E^\prime}$ and ${E^\prime}{O^\prime}={O^\prime}$, and we get
        \begin{align*}
            EO &= \frac{1}{4}(2{O^\prime} + 2{e_m}{E^\prime}) = \frac{1}{2}({O^\prime} + {e_m}{E^\prime}) = O.
        \end{align*}
        
        Once we have that $O^2 = -E, E^2 = E$ and $EO=O$, the result follows if one applies operation of conjugation $\conj{(.)}{B}$ for $B\subseteq\{2,3,\hdots,n\}$ on both sides of each one of these equations.
\end{proof}

\begin{proof}[Proof of 2]
    Fix $k\in\firstn{n}$. We will first show that ${e_k}E = O$ and ${e_k}O = -E$. Like in the proof above, let 
        \begin{align*}
            f^\prime &= \frac{1}{2^{n-2}}\left(\sum_{A\subseteq\firstn{n}\setminus\{k\}}\right),\\
            E^\prime &= \,\sum_{k=0}^{\lfloor\frac{n-1}{2}\rfloor}(-1)^{k}\;\proj{f^\prime}{2k} = \proj{f^\prime}{0}-\proj{f^\prime}{2}+\proj{f^\prime}{4}-\cdots,\\
            O^\prime &= \;\sum_{k=0}^{\lfloor\frac{n-1}{2}\rfloor}(-1)^{k}\;\proj{f^\prime}{2k+1} = \proj{f^\prime}{1}-\proj{f^\prime}{3}+\proj{f^\prime}{5}-\cdots.
        \end{align*}
        Note that then 
        \begin{align*}
            f &= \frac{1}{2}\left(f^\prime + {e_k}f^\prime\right),\hspace{10pt} E = \frac{1}{2}\left(E^\prime - {e_k}O^\prime\right),\hspace{10pt} O = \frac{1}{2}\left(O^\prime + {e_k}E^\prime\right).
        \end{align*}
        Now, 
        \begin{align*}
            {e_k}E &= {e_k}\frac{1}{2}\left(E^\prime - {e_k}O^\prime\right) = \frac{1}{2}({e_k}{E^\prime} + {O^\prime}) = \frac{1}{2}({O^\prime} + {e_k}{E^\prime}) = O.
        \end{align*}
        And 
        \begin{align*}
            {e_k}O &= {e_k}\frac{1}{2}\left(O^\prime + {e_k}E^\prime\right) = \frac{1}{2}({e_k}{O^\prime} - {E^\prime}) = -\frac{1}{2}({E^\prime} - {e_k}{O^\prime}) = -E.
        \end{align*}
        Now, since $\conj{e_k}{B} = \left\{
        \begin{aligned}
            &e_k,& \mbox{if $k \notin B$};\\
            &-e_k,& \mbox{if $k \in B$,}
        \end{aligned}\right.$  the result follows.
\end{proof}
\begin{proof}[Proof of 3]
    Since $E_B = \conj{E}{B}$, when we add all $E_B$s together, all the terms in which a generator appears cancel out because in the sum, there are an equal number of terms involving $e_A$s with a plus sign and terms involving $e_A$s with a minus sign. Therefore, in the sum $\sum_{B\subseteq\firstn{n}} E_B$, only the grade $0$ terms survive and their sum is equal to 1 because each $E_B$ contributes $\frac{1}{2^{n-1}}$ and there are $2^{n-1}$ $E_B$s. 
\end{proof}
\begin{lem}\label{Lemma: orthogonality in K(0,n)}
    Let $n\ge2$. Let $B$ be a non-empty subset of $\{2,3,\hdots,n\}$, then $OO_B = 0, OE_B = 0, EO_B = 0$ and $EE_B = 0$.
\end{lem}
\begin{proof}
    We will prove this using mathematical induction on $n$. Consider the case $n=2$. Then $O = \frac{1}{2}(e_1 + e_2)$, $E = \frac{1}{2}(1 - e_{12})$. The only non-empty subset that we have to consider is $\{2\}$. Then
    \begin{align*}
        O_{\{2\}} = \conj{O}{2} = \frac{1}{2}(e_1 - e_2),\hspace{10pt} E_{\{2\}} = \conj{E}{2} = \frac{1}{2}(1 + e_{12}).
    \end{align*}
    
    Now, $OO_{\{2\}} = \frac{1}{4}(e_1 + e_2)(e_1 - e_2) = \frac{1}{4}(1-{e_1}{e_2}+{e_2}{e_1}-1)=0$, $EO_{\{2\}} = \frac{1}{4}(1 - e_{12})(e_1 - e_2) = \frac{1}{4}(e_1 + e_2 - e_2 - e_1) = 0$, $OE_{\{2\}} = \frac{1}{4}(e_1 + e_2)(1 + e_{12}) = \frac{1}{4}(e_1 - e_2 + e_2 -e_1) = 0$ and $EE_{\{2\}} = \frac{1}{4}(1 - e_{12})(1 + e_{12}) = \frac{1}{4}(1 + e_{12} - e_{12} -1 = 0)$. This is our base case.

    Our induction hypothesis is that the result holds for $n=m$. Assuming this, we will prove that the result holds for $n=m+1$. Let $B\subseteq\firstn{m+1}$ such that $|B|>0$. Then we can express $B$ as 
    $$B = B^\prime\sqcup\{k\}$$
    where $k\in B$ and $B^\prime = B\setminus\{k\}$. Let 
        \begin{align*}
            f^\prime &= \frac{1}{2^{m}}\left(\sum_{A\subseteq\firstn{n}\setminus\{k\}}\right),\\
            E^\prime &= \,\sum_{k=0}^{\lfloor\frac{n-1}{2}\rfloor}(-1)^{k}\;\proj{f^\prime}{2k} = \proj{f^\prime}{0}-\proj{f^\prime}{2}+\proj{f^\prime}{4}-\cdots,\\
            O^\prime &= \;\sum_{k=0}^{\lfloor\frac{n-1}{2}\rfloor}(-1)^{k}\;\proj{f^\prime}{2k+1} = \proj{f^\prime}{1}-\proj{f^\prime}{3}+\proj{f^\prime}{5}-\cdots.
        \end{align*}
        
        Then 
        \begin{align*}
            f = \frac{1}{2}\left(f^\prime + {e_k}f^\prime\right),\hspace{10pt}
            E = \frac{1}{2}\left(E^\prime - {e_k}O^\prime\right),\hspace{10pt}
            O = \frac{1}{2}\left(O^\prime + {e_k}E^\prime\right).
        \end{align*}
        
        First, consider $\conj{O}{k}\conj{O}{B^\prime}$:
        \begin{align*}
            \conj{O}{k}\conj{O}{B^\prime} &= \frac{1}{2}\left(O^\prime - {e_k}E^\prime\right)\conj{\left(\frac{1}{2}\left(O^\prime + {e_k}E^\prime\right)\right)}{B^\prime}\\
            &= \frac{1}{4}(O^\prime - {e_k}E^\prime)\left(\conj{(O^\prime)}{B^\prime} + {e_k}\conj{(E^\prime)}{B^\prime}\right)\\
            &= \frac{1}{4}\left({O^\prime}\conj{(O^\prime)}{B^\prime} + {e_k}{O^\prime}\conj{(E^\prime)}{B^\prime} -{e_k}{E^\prime}\conj{(O^\prime)}{B^\prime} +{E^\prime}\conj{(E^\prime)}{B^\prime}\right).
        \end{align*}
        
        By our induction hypothesis, all of the terms ${O^\prime}\conj{(O^\prime)}{B^\prime}$, ${O^\prime}\conj{(E^\prime)}{B^\prime}$, ${E^\prime}\conj{(O^\prime)}{B^\prime}$ and ${E^\prime}\conj{(E^\prime)}{B^\prime}$ are zero and therefore $\conj{O}{k}\conj{O}{B^\prime} = 0$. Note that $OO_B = O\conj{O}{B} = \conj{(\conj{O}{k}\conj{O}{B^\prime})}{k}$. Since $\conj{O}{k}\conj{O}{B^\prime} = 0$, we have that $OO_B = 0$.
        
        Similarly, using the same idea, we also get that $OE_B = 0, EO_B = 0$ and $EE_B = 0$.
\end{proof}

Let $U\in\dl{0}{n}$. Then for each $A \subseteq\firstn{n}$, define linear map $\gamma_A:\dl{0}{n}\to\C$ by $\gamma_A(e_C) := (-1)^{|A\cap C|}i^{|C|}$, i.e., we replace generators in $e_C$ by $-i$ if their index is present in $A$ and $+i$ if their index is absent in $A$. We have defined $\gamma_A$ on the basis elements and one then extends it to all of $\dl{0}{n}$ via linearity. Define $\zeta_A^R,\zeta_A^I:\dl{0}{n}\to\R$ by $\zeta_A^R(U) := \Re(\gamma_A(U))$, $\zeta_A^I(U) := \Im(\gamma_A(U))$ where $\Re(z)$ denotes the real part and $\Im(z)$ denotes the imaginary part of a complex number $z$.

\begin{thm}\label{Theorem: real irred reps are the coefficients K(0,n)}

Let $U\in\dl{0}{n}$, $C\subseteq\{2,3,\hdots,n\}$. Then $UE_C = 
    {\zeta_C^R}(U)E_C + {\zeta_C^I}(U)O_C$.
\end{thm}

\begin{proof} 
    Let $B\subseteq\{2,3,\hdots,n\}$. Recall from Lemma \ref{Lemma: first properties of special basis in K(0,n)} that for all $k\in\firstn{n}$,
    \begin{eqnarray*}
        {e_k}{E_B} = 
        \begin{cases}
            -O_B,& \mbox{if $k\in B$,}\\
            O_B,& \mbox{if $k\notin B$,}
        \end{cases} \qquad 
        {e_k}{O_B} =
        \begin{cases}
            E_B,& \mbox{if $k\in B$,}\\
            -E_B,& \mbox{if $k\notin B$.}
        \end{cases}
    \end{eqnarray*}
    A first consequence of this is that if $D\not\subseteq C$, then    
    \begin{eqnarray*}
    {e_D}{O_C} = 
    \begin{cases}
        (-1)^\frac{|D|+1}{2}{E_C},& \mbox{if $|D|$ is odd,}\\
        (-1)^\frac{|D|}{2}{O_C},& \mbox{if $|D|$ is even,}
    \end{cases}
    {e_D}{E_C} =
    \begin{cases}
        (-1)^\frac{|D|-1}{2}{O_C},& \mbox{if $|D|$ is odd,}\\
        (-1)^\frac{|D|}{2}{E_C},& \mbox{if $|D|$ is even.}
    \end{cases}
    \end{eqnarray*}
        And if $D\subseteq C$, then
    \begin{eqnarray*}
    {e_D}{O_C} = 
    \begin{cases}
        (-1)^\frac{|D|-1}{2}{E_C},& \mbox{if $|D|$ is odd,}\\
        (-1)^\frac{|D|}{2}{O_C},& \mbox{if $|D|$ is even,}
    \end{cases}
    {e_D}{E_C} =
    \begin{cases}
        (-1)^\frac{|D|+1}{2}{O_C},& \mbox{if $|D|$ is odd,}\\
        (-1)^\frac{|D|}{2}{E_C},& \mbox{if $|D|$ is even.}
    \end{cases}
    \end{eqnarray*}
     Let $A = \{a_1,a_2,\hdots,a_m\} \subseteq\firstn{n}$. We use the above formulas iteratively use this to first find ${e_A}{E_B}$. First, note that $e_A = {e_{A\setminus C}} \,{e_{A\cap C}}$. Now, ${e_{A\cap C}}{E_C} = \left\{
        \begin{aligned}
            &(-1)^\frac{|A\cap C|+1}{2}{O_C},& \mbox{if $|A\cap C|$ is odd,}\\
            &(-1)^\frac{|A\cap C|}{2}{E_C},& \mbox{if $|A\cap C|$ is even}
        \end{aligned}
        \right.$ and
        
        \begin{align*}
            &{e_A}{E_C} = {e_{A\setminus C}}{e_{A\cap C}}= \left\{
            \begin{aligned}
                &(-1)^\frac{|A\cap C|+1}{2}{e_{A\setminus C}}{O_C},& \mbox{if $|A\cap C|$ is odd,}\\
                &(-1)^\frac{|A\cap C|}{2}{e_{A\setminus C}}{E_C},& \mbox{if $|A\cap C|$ is even}
            \end{aligned}
            \right.\\
            &= \left\{
            \begin{aligned}
                &{(-1)^\frac{|A\cap C|+1}{2}}{(-1)^\frac{|A\setminus C|+1}{2}}{E_C},& \mbox{if $|A\cap C|$ is odd and $|A\setminus C|$ is odd,}\\
                &{(-1)^\frac{|A\cap C|+1}{2}}{(-1)^\frac{|A\setminus C|}{2}}{O_C},& \mbox{if $|A\cap C|$ is odd and $|A\setminus C|$ is even,}\\
                &{(-1)^\frac{|A\cap C|}{2}}{(-1)^\frac{|A\setminus C|-1}{2}}{O_C},& \mbox{if $|A\cap C|$ is even and $|A\setminus C|$ is odd,}\\
                &{(-1)^\frac{|A\cap C|}{2}}{(-1)^\frac{|A\setminus C|}{2}}{E_C},& \mbox{if $|A\cap C|$ is even and $|A\setminus C|$ is even}
            \end{aligned}
            \right.\\
            &= \left\{
            \begin{aligned}
                &(-1)^{{\frac{|A|}{2}}+1}{E_C},& \mbox{if $|A\cap C|$ is odd and $|A\setminus C|$ is odd,}\\
                &{(-1)^\frac{|A|+1}{2}}{O_C},& \mbox{if $|A\cap C|$ is odd and $|A\setminus C|$ is even,}\\
                &{(-1)^\frac{|A|-1}{2}}{O_C},& \mbox{if $|A\cap C|$ is even and $|A\setminus C|$ is odd,}\\
                &{(-1)^\frac{|A|}{2}}{E_C},& \mbox{if $|A\cap C|$ is even and $|A\setminus C|$ is even.}
            \end{aligned}
            \right.
        \end{align*}
    This means that if $|A|$ is even, then ${e_A}{E_C} = \pm{E_C}$ and if $|A|$ is odd, then ${e_A}{E_C} = \pm{O_C}$. Note that we get the signs given above for the basis elements $e_A$ if we replace the generators in $e_A$ common to $e_C$ with $-i$ and replace remaining generators in $e_A$ by $i$ and think of $E_C$ as $1$ and $O_C$ as $i$. Therefore, if we take a general element $U = \sum_{A\subseteq\firstn{n}}{{a_A}{e_A}} \in \dl{0}{n}$, then $UE_C = {{\zeta_C}^R}(U)E_C + {{\zeta_C}^I}(U)O_C$. This completes the proof.
\end{proof}

\begin{thm}\label{Theorem: f_B^1s, f_B^is is an orthonormal basis}
    Let $A,B \subseteq\{2,3,\hdots,n\}$, then\\
    
    ${E_A}{E_B} = \left\{
    \begin{aligned}
        &E_A,& \mbox{if $A=B$,}\\
        &0,& \mbox{if $A\ne B$,}
    \end{aligned}\right.$\quad
    ${E_A}{O_B} = \left\{
    \begin{aligned}
        &O_A,& \mbox{if $A=B$,}\\
        &0,& \mbox{if $A\ne B$,}
    \end{aligned}\right.$\\
    
    ${O_A}{O_B} = \left\{
    \begin{aligned}
        &-E_A,& \mbox{if $A=B$,}\\
        &0,& \mbox{if $A\ne B$.}
    \end{aligned}\right.$\\
    
    Further, $\{E_B, O_B \;|\; B\subseteq\{2,3,\hdots,n\}\}$ is a basis for $\dl{0}{n}$.
\end{thm}
\begin{proof}
    Note that the first part is a direct consequence of Lemma \ref{Lemma: first properties of special basis in K(0,n)} and \ref{Lemma: orthogonality in K(0,n)}. 
    
    Now, we show that the set $\{E_B, O_B \;|\; B\subseteq\{2,3,\hdots,n\}\}$ is a basis for $\dl{0}{n}$. First we will show linear independence. Let $d_B^O$ and $d_B^E \in\R$ be such that $\sum_{B\subseteq\{2,3,\hdots,n\}}{{d_B^O}{O_B} + {d_B^E}{E_B}} = 0$. Multiplying both sides of the above equation by $E_A$, we get ${d_A^O}{O_A} + {d_A^E}{E_A} = 0$. Because $O_A$ and $E_A$ comprise of elements of different grades, we get ${d_A^O} = {d_A^E} = 0$. Since $A\subseteq\{2,3,\hdots,n\}$ was arbitrary, we get $d_A = 0$ for all $A\subseteq\{2,3,\hdots,n\}$. This shows linear independence.

    Lastly, we show that $\{E_B, O_B \;|\; B\subseteq\{2,3,\hdots,n\}\}$ spans $\dl{0}{n}$. Let $U\in\dl{0}{n}$. Then, using Theorem \ref{Theorem: real irred reps are the coefficients K(0,n)} and result 3 of Lemma \ref{Lemma: first properties of special basis in K(0,n)}, we have that 
    \begin{multline}\label{basis change formula for K_{0,n}}
        U = U\cdot1 = U\cdot\left(\sum_{B\subseteq\{2,3,\hdots,n\}}E_B\right) = \sum_{B\subseteq\{2,3,\hdots,n\}} UE_B\\
        = \sum_{B\subseteq\{2,3,\hdots,n\}}{{{\zeta_B}^R}(U)E_B + {{\zeta_B}^I}(U)O_B}.
    \end{multline}
    
    This shows that the set $\{E_B, O_B \;|\; B\subseteq\{2,3,\hdots,n\}\}$ spans $\dl{0}{n}$.
\end{proof}

We need an enumeration of the set of subsets of $\{2,3,\hdots,n\}$. We consider the enumeration given in the paper \cite{A.Acus} like we did previously for $\dl{n}{0}$. We arrange the subsets of $\{2,3,\hdots,n\}$ in lexicographic order and then enumerate them. Let us denote the index of a subset $B\subseteq\{2,3,\hdots,n\}$ according to this enumeration by $i(B)$. We index the empty set $\{\}$ as the first set. Let $b_k^1 = (0,0,\hdots,0,1,0,\hdots,0)$, $b_k^i=(0,0,\hdots,0,i,0,\hdots,0) \in \C^{2^{n-1}}$ be as before.

\begin{thm}\label{Theorem: isomorphism between K_{0,n} and direct sum of Cs}
    The identification $E_B \leftrightarrow b_{i(B)}^1, O_B \leftrightarrow b_{i(B)}^i$ is an isomorphism between $\dl{0}{n}$ and $\C^{2^{n-1}}$.
\end{thm}
\begin{proof}
    The proof is a direct consequence of Theorem \ref{Theorem: f_B^1s, f_B^is is an orthonormal basis} and the discussion given in Section \ref{subsection: direct sums of Rs and Cs}.
\end{proof} 
\begin{rem}
    As mentioned in Remark \ref{remark about convention of E_Bs}, we could have considered the basis $\{E_B, O_B \;|\; B\subseteq\firstn{n}\setminus\{k\}\}$ for any fixed $k\in\firstn{n}$. In that case, Theorem \ref{Theorem: isomorphism between K_{0,n} and direct sum of Cs} holds true but we will have $B\subseteq\firstn{n}\setminus\{k\}$ instead of $B\subseteq\{2,3,\hdots,n\}$. For $B\subseteq\firstn{n}\setminus\{k\}$, $UE_B$ would still equal $\zeta_B^R(U)E_B + \zeta_B^I(U)O_B$ but the coordinates of image of $U\in\dl{0}{n}$ under the isomorphism would be different than the coordinates we got for $k=1$.
\end{rem}

\section{Discussion regarding multicomplex spaces}\label{discussion}
In this paper, we introduced commutative analogues of Clifford algebras $\dl{p}{q}$. We showed that the multicomplex numbers are a special case of these algebras, being isomorphic to $\dl{0}{n}$. As shown in this paper, $\dl{p}{q}$ has a tensor product decomposition and a direct sum decomposition and therefore multicomplex spaces can also be decomposed as a tensor product of copies of complex numbers and as a direct sum of copies of complex numbers. The latter fact is well known in the multicomplex spaces literature and goes by the name of idempotent representation of multicomplex spaces. Price, in his book \cite{Price}, presents an elegant iterative way of finding idempotents in multicomplex spaces. Also check \cite{InvolutionsMC} for the discussion about idempotent decomposition (this paper also introduced involutions in multicomplex spaces which can be though of as a generalisation of operations of conjugation). This approach of finding idempotents was continued in \cite{MulticomplexIdeals} where the authors also introduced operations of conjugation which allows them to express the idempotents as conjugates of one idempotent like we have done. We were not aware of the literature on multicomplex spaces when we started studying the algebras $\dl{p}{q}$ and present, in this paper, our approach to find idempotents. We found out explicitly an isomorphism between multicomplex space $\dl{0}{n}$ and $\C^{2^{n-1}}$ and in order to do that we constructed idempotents $E_B$ for $B\subseteq\{2,3,\hdots,n\}$. Check Section \ref{subsection: special basis in K_{0,n}} for details. Working with the direct sum decomposition of multicomplex spaces is more convenient and we hope that the explicit isomorphism would enable one to go back and forth between the usual definition of multicomplex spaces and their direct sum decomposition allowing them to prove results in a simpler manner. For example, it had been conjectured in \cite{FrenchMSThesis} that there are exactly $2^{2^{n-1}}$ idempotents in $\dl{0}{n}$ which is evident from the fact that $\dl{0}{n}\cong\C^{2^{n-1}}$ and one can find them explicitly using the isomorphism presented in Theorem \ref{Theorem: isomorphism between K_{0,n} and direct sum of Cs}. Corollary 2.3 in \cite{InvolutionsMC} also becomes apparent if one makes note of the direct sum decomposition in Theorem \ref{Theorem: isomorphism between K_{0,n} and direct sum of Cs}.

Something which is evident but not talked about is that multicomplex spaces $\dl{0}{n}$ have a tensor product decomposition as described in this paper. This decomposition encapsulates the fact that one can represent elements in multicomplex space $\dl{0}{n}$ as elements of $\dl{0}{k}$ but instead of having real coefficients, their coefficients are elements of $\dl{0}{n-k}$ for $k\in\firstn{n}$ as the property \eqref{Tensor decomposition Property 1}. One can also use the tensor product decomposition to their advantage when considering the regular representation of $\dl{0}{n}$ which Price calls Cauchy--Riemann matrices. This is because the regular representation of tensor product of algebras is the tensor product representation of the regular representation of individual algebras. Therefore, the Cauchy--Riemann matrices are Kronecker products of matrices representing complex numbers. This greatly simplifies the proving many results which were previously proven using row reduction methods in Price's book \cite{Price}.

\section{Conclusions}\label{section_conclusions}
In this paper we define and discuss commutative analogues of Clifford algebras $\dl{p}{q}$. It is to be noted that commutativity made life simpler by giving a tensor product decomposition \eqref{tensor product decomposition result} for the algebras. We showed that $\dl{p}{q}$ are either isomorphic to $\dl{n}{0}$ or isomorphic to $\dl{0}{n}$ -- the multicomplex numbers. Next, we introduced operations of conjugations in $\dl{p}{q}$, which are essentially a generalisation of complex conjugation in $\C$. Lastly, we presented a special basis in $\dl{p}{q}$ which allowed them to be expressed as a direct sum of $\R$s or $\C$s depending on which isomorphism class the algebra is in. We summarise the main results below.
\begin{enumerate}
    \item Tensor product decomposition (Theorem \ref{tensor product decomposition theorem}):
        \begin{equation*}
            \dl{p}{q} \cong \underbrace{\dl{1}{0} \otimes \dl{1}{0} \cdots \otimes \dl{1}{0}}_{p \; \text{times}} \otimes \underbrace{\dl{0}{1} \otimes \dl{0}{1} \cdots \otimes \dl{0}{1}}_{q \; \text{times}}.
        \end{equation*}
    This follows because one can understand tensor product as extension of scalars. More explicitly, the isomorphism is given by mapping $k$th generator of $\dl{p}{q}$ to the tensor product of $n$ 1s except the generator of a $\dl{1}{0}$ or $\dl{0}{1}$ at the $k$th place.
    \item Isomorphism classes in $\dl{p}{q}$ (Theorem \ref{thm: isomorphism classes}). Let $p,q\in\Z_{\ge0}$ and $n = p+q$.
        \begin{align*}
            \text{If } q\ge1, \text{ then } \dl{p}{q} \cong \dl{0}{n} \text{ and } \dl{0}{n} \not\cong \dl{n}{0}.
        \end{align*}
    \item Direct sum decomposition of $\dl{n}{0}$ (Theorem \ref{Theorem: isomorphism between K_{n,0} and direct sum of Rs}):
        \begin{equation*}
            \dl{n}{0} \cong \R^{2^n} = \underbrace{\R\oplus\R\oplus\cdots\oplus\R}_{2^n \,\, \text{times}}.
        \end{equation*}
    Let 
    \begin{equation*}
        f = \frac{1}{2^n}\left(\sum_{A\subseteq\firstn{n}}e_A\right) \in \dl{n}{0}.
    \end{equation*}
    Define $f_B := \conj{f}{B}$, where $\conj{(.)}{B}$ denotes the $B$th conjugate of an element (Section \ref{Operations of conjugation} discusses operations of conjugation). These elements are such that for all $A,B\subseteq\firstn{n}$,
    $${f_A}{f_B} = \left\{
        \begin{aligned}
            &f_A,& \mbox{if $A=B$},\\
            &0,& \mbox{if $A\ne B$.}
        \end{aligned}\right.$$
    Existence of such elements directly leads to the above isomorphism between $\dl{n}{0}$ and $\R^{2^n}$ where one maps these elements to the standard basis of $\R^{2^n}$. Refer to Section \ref{subsection: special basis for K(n,0)} for details.

    \item Direct sum decomposition of $\dl{0}{n}$ (Theorem \ref{Theorem: isomorphism between K_{0,n} and direct sum of Cs}):
    \begin{equation*}
            \dl{0}{n} \cong \C^{2^{n-1}} = \underbrace{\C\oplus\C\oplus\cdots\oplus\C}_{2^{n-1} \,\, \text{times}}.
        \end{equation*}
    Let 
    \begin{align*}
        f &:= \frac{1}{2^{n-1}}\left(\sum_{A\subseteq\firstn{n}}{e_A}\right)\in\dl{0}{n},\\
        E &:= \,\sum_{k=0}^{\lfloor\frac{n}{2}\rfloor}(-1)^{k}\;\proj{f}{2k} = \proj{f}{0}-\proj{f}{2}+\proj{f}{4}-\cdots,\\
        O &:= \;\sum_{k=0}^{\lfloor\frac{n}{2}\rfloor}(-1)^{k}\;\proj{f}{2k+1} = \proj{f}{1}-\proj{f}{3}+\proj{f}{5}-\cdots.
    \end{align*}
    Define $E_B = \conj{E}{B}$ and $O_B = \conj{O}{B}$ for all $B\subseteq\{2,3,\hdots.n\}$,  where $\conj{(.)}{B}$ denotes the $B$th conjugate of an element (Section \ref{Operations of conjugation} discusses operations of conjugation). These elements are such that for all $A,B \subseteq\{2,3,\hdots,n\}$,\\

    ${E_A}{E_B} = \left\{
    \begin{aligned}
        &E_A,& \mbox{if $A=B$,}\\
        &0,& \mbox{if $A\ne B$,}
    \end{aligned}\right.$\quad
    ${E_A}{O_B} = \left\{
    \begin{aligned}
        &O_A,& \mbox{if $A=B$,}\\
        &0,& \mbox{if $A\ne B$,}
    \end{aligned}\right.$ and\\
    
    ${O_A}{O_B} = \left\{
    \begin{aligned}
        &-E_A,& \mbox{if $A=B$,}\\
        &0,& \mbox{if $A\ne B$.}
    \end{aligned}\right.$
    
    \noindent Existence of such elements directly leads to the above isomorphism between $\dl{0}{n}$ and $\C^{2^{n-1}}$ where one maps these elements to the standard basis of $\C^{2^{n-1}}$, thinking of it as an $\R$-vector space. Refer to Section \ref{subsection: special basis in K_{0,n}} for details.
\end{enumerate}

We hope that the decompositions provided in this paper will aid further study of commutative analogues of Clifford algebras, in particular, multicomplex spaces and commutative quaternions. We were motivated to consider such algebras because the subalgebra ${\rm{Span}_\R}\{1,e_{1256},e_{1346},e_{2345}\} \subseteq \cl{6}{0}$ described in \cite{A.Acus} causes trouble with respect to finding a explicit formula for multiplicative inverse and is isomorphic to $\dl{2}{0}$. We hope to better understand the problem of finding multiplicative inverses in Clifford algebras next. In the next paper \cite{DetInCommAnalCliffAlg}, we solve the problem of finding multiplicative inverses in commutative analogues of Clifford algebras by introducing the notion of determinant in them,  which involve only operations of conjugation and do not involve matrix operations.

\section*{Acknowledgment}
The first author deeply thanks Prof. Dmitry Shirokov for giving him the opportunity to work with him. The first author also wishes to thank his friends, especially Chitvan Singh and Ipsa Bezbarua, who were always there for him.

The authors are grateful to the anonymous reviewers for their careful reading of the paper and helpful comments on how to improve the presentation.

The article was prepared within the framework of the project “Mirror Laboratories” HSE University “Quaternions, geometric algebras and applications”.
\medskip

\noindent{\bf Data availability} Data sharing not applicable to this article as no datasets were generated or analyzed during the current study.

\medskip

\noindent{\bf Declarations}\\
\noindent{\bf Conflict of interest} The authors declare that they have no conflict of interest.

\bibliographystyle{spmpsci}

\end{document}